\documentclass{amsart}
\usepackage[greek,english]{babel}
\usepackage{amssymb,graphicx,microtype}
\newcommand{\Aut}{\operatorname{Aut}}
\renewcommand{\Vert}{\operatorname{V}}
\newcommand{\Edge}{\operatorname{E}}
\renewcommand{\star}{\operatorname{star}}

\newcommand{\Alpha}{\text{\greektext A}}
\newcommand{\Beta}{\text{\greektext B}}
\renewcommand{\Gamma}{\text{\greektext G}}
\renewcommand{\Lambda}{\text{\greektext L}}
\renewcommand{\Phi}{\text{\greektext F}}
\renewcommand{\Xi}{\text{\greektext X}}
\newcommand{\head}{\partial_1}
\newcommand{\tail}{\partial_0}
\newcommand{\Z}{\mathbb Z}
\newtheorem{theorem}{Theorem}[section]
\newtheorem*{theorem*}{Theorem}

\newtheorem{lemma}[theorem]{Lemma}

\theoremstyle{definition}
\newtheorem{definition}[theorem]{Definition}
\newtheorem{example}[theorem]{Example}
\newtheorem{remark}[theorem]{Remark}

\begin{document}
\title{On Leighton's graph covering theorem}
\author{Walter D. Neumann}
\begin{abstract}We give %simple proofs of versions of Leighton's graph
  %covering theorem as well as 
  short expositions of both Leighton's proof and the Bass-Kulkarni
  proof of Leighton's graph covering theorem, in the context of
  colored graphs. We discuss a further generalization, needed
  elsewhere, to ``symmetry-restricted graphs.''  We can prove it in
  some cases, for example, if the ``graph of colors'' is a tree, but
  we do not know if it is true in general. We show that Bass's
  Conjugation Theorem, which is a tool in the Bass-Kulkarni approach,
  does hold in the symmetry-restricted context.
\end{abstract}
\maketitle
%\section{Introduction}
Leighton's graph covering theorem says:
\begin{theorem*}[Leighton \cite{leighton}]
Two finite graphs
  which have a common covering have a common finite covering.
\end{theorem*}
It answered a conjecture of Angluin and Gardiner who had proved the
case that both graphs are $k$--regular \cite{reg}.  Leighton's proof
is short (two pages), but has been considered by some to lack
transparency. It was reframed in terms of Bass-Serre theory by Bass
and Kulkarni \cite{BK, B}, expanding its length considerably
but providing group-theoretic tools which have other uses.

The general philosophy of the Bass-Kulkarni proof is that adding more
structure helps. Let us illustrate this by giving a very short
proof of Angluin and Gardiner's original $k$--regular case.

We assume all graphs considered are connected. ``Graph'' will thus
mean a connected $1$--complex.  ``Covering'' means covering space in
the topological sense. Two graphs are isomorphic if they are
isomorphic as $1$--complexes (i.e., homeomorphic by a map which is
bijective on the vertex and edge sets).  We want to show that if $G$
and $G'$ are finite $k$-regular graphs (i.e., all vertices have
valence $k$) then they have a common finite covering.
\begin{proof}[Proof of the $k$--regular case]
  Replace $G$ and $G'$ by oriented ``fat
  graphs''---thicken 
edges to rectangles of length $10$ and width $1$, say, and replace
vertices by regular planar $k$-gons of side length $1$, to which the
\begin{figure}[ht]
    \centering
\includegraphics[width=.3\hsize]{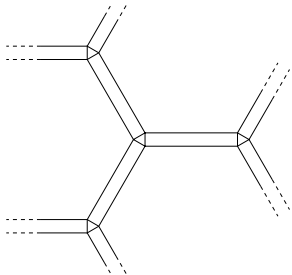}
    \caption{3-regular fat graph} 
    \label{fig:1}
\end{figure}
rectangles are glued at their ends (see Fig.~\ref{fig:1}; the underlying
space of the fat graph is often required to be orientable as a
2-manifold but we don't need this). $G$ and $G'$ both
have universal covering the $k$--regular fat tree $T_k$, whose isometry
group $\Gamma$ acts properly discretely (the orbit space $T_k/\Gamma$
is the $2$--orbifold pictured in Fig.~\ref{fig:2}).  The covering
transformation groups for the coverings $T_k\to G$ and $T_k\to G'$ are
finite index subgroups $\Lambda$ and $\Lambda'$ of $\Gamma$. The
quotient $T_k/(\Lambda\cap \Lambda')$ is the desired common finite
covering of $G$ and $G'$.
\end{proof}
\begin{figure}[ht]
    \centering
\includegraphics[width=.4\hsize]{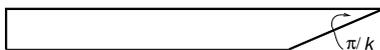}
    \caption{$T_k/\Aut(T_k)$; bold edges are mirror edges.} 
    \label{fig:2}
\end{figure}

We return now to unfattened graphs. In addition to the simplicial view
of graphs, it is helpful to consider in parallel a 
combinatorial point of view, in which an edge of an undirected graph
consists of a pair $(e,\bar e)$ of directed edges.  From this point of
view a graph $G$ is defined by a vertex set $\Vert(G)$ and directed
edge set $\Edge(G)$, an involution $e\mapsto \bar e$ on $\Edge(G)$,
and maps $\tail$ and $\head$ from $\Edge(G)$ to $\Vert(G)$ satisfying
$\tail\bar e=\head e$ for all $e\in \Edge(G)$. One calls $\tail e$ and
$\head e$ the \emph{tail} and \emph{head} of $e$.

The combinatorial point of view is especially convenient for quotients
of graphs by groups of automorphisms: if a group of automorphisms
inverts some edge, the corresponding edge in the quotient graph will
be a directed loop (an edge satisfying $e=\bar e$; in the simplicial
quotient this is a ``half-edge''---an orbifold with
underlying space an interval having a vertex at one end and a mirror
at the other).
% Thus we allow directed loops in the graphs we consider.

A \emph{coloring} of a graph $G$ will mean a graph-homomorphism of $G$
to a fixed \emph{graph of colors}. The vertex and edge sets of this
graph are the \emph{vertex-colors} and \emph{edge-colors}
respectively. By a graph-homomorphism of a colored graph we always
mean one which preserves colors; in particular, covering maps should
preserve colors, and for a colored graph $G$, $\Aut(G)$ will always
mean the group of colored graph automorphisms.

It is an exercise to derive from Leighton's theorem the version for
colored graphs.  But it is also implicit in Leighton's proof, so we
will describe this in Section \ref{sec:leighton}.  This paper was
motivated by the desire in \cite{BN} of a yet more general version,
which we describe in Section \ref{sec:symmetry-restricted}, and prove
in a special case in Section \ref{sec:symmetry-restricted-again},
using the Bass-Kulkarni approach, which we expose in Section
\ref{sec:bk}.  

The \emph{universal covering} $\tilde G$ of a colored graph $G$ is its
universal covering in the topological sense, i.e., of the underlying
undirected graph as a simplicial complex. 
%(or orbifold, if we allow half-edges). 
This is a colored tree, with the coloring induced from that of $G$. If
$\Aut(\tilde G)$ does not act transitively on the set of vertices or
edges of $\tilde G$ of each color, we can refine the colors to make it
so, by replacing the graph of colors by the \emph{refined graph of
  colors} $C:=\tilde G/\Aut(\tilde G)$. This does not change
$\Aut(\tilde G)$. We will usually use refined colors, since graphs
which have a common covering have the same universal covering and
therefore have the same refined colors.

\section{Leighton's theorem for colored graphs}
\label{sec:leighton}
We give Leighton's proof, mildly modified to clarify its structure and
to make explicit the fact that it handles colored graphs.
To ease comparison with his version, we have copied some of his notation.
\begin{theorem}\label{th:colored}
  Two finite colored graphs $G$ and $G'$ which have a common covering
  have a common finite covering.
\end{theorem}
\begin{proof} 
  We can assume we are working with refined
  colors, so $C:=\widetilde G/\Aut(\widetilde G)$ is our graph of
  colors. We denote the sets of vertex and edge colors by
  $I=\Vert(C)$, $K=\Edge(C)$.  For $k\in K$ we write $\partial k=ij$
  if $\tail k=i$ and $\head k=j$. An \emph{$i$--vertex} is one with
  color $i$ and a \emph{$k$--edge} is one with color $k$.

  Denote by $n_i$ and $m_k$ the numbers of $i$--vertices and $k$-edges
  of $G$. For $k\in K$ with $\partial k=ij$ denote by $r_{k}$ the
  number of $k$--edges from any fixed $i$--vertex $v$ of $G$.
  Clearly $$n_ir_{k}=m_{k}=m_{\bar k}=n_jr_{\bar k}\,.$$ Let $s$ be a
  common multiple of the $m_{k}$'s. Put $a_i:=\frac s{n_i}$ and
  $b_{k}:=\frac{a_i}{r_{k}}=\frac s{m_{k}}$. Then
\begin{equation}
  \label{eq:1}
b_k=\frac {a_i}{r_{k}}=\frac {a_j}{r_{\bar k}}=b_{\bar k}\,.
\end{equation}

The $a_i$ and $b_{k}$ can be defined by equations \eqref{eq:1},
without reference to $G$ (or $G'$). For if positive integers $a_i$
($i\in I$) and $b_{k}$ ($k\in K$) satisfy \eqref{eq:1} whenever
$\partial k=ij$, then $a_i=\frac{r_{k}a_j}{r_{\bar k}}$, so
$n_ia_i=\frac{n_ir_{k}a_j}{r_{\bar k}}=\frac{n_jr_{\bar k}a_j}{r_{\bar
    k}}=n_ja_j$, so $s:=n_ia_i$ is independent of $i$. This $s$ is
divisible by every $m_{k}$ and $a_i=\frac s{n_i}$.

For $i\in I$ choose a set $\Alpha_i$ of size $a_i$. For $k\in K$
choose a group $\Pi_{k}$ of size $r_{k}$,  a set
$\Beta_{k}=\Beta_{\bar k}$ of size $b_{k}$ and a bijection
$\phi_{k}\colon \Pi_{k}\times \Beta_{k}\to \Alpha_{\tail k}$.

For each $i$--vertex $v$ of $G$ choose a bijection $\psi_{vk}$ of the
set of $k$--edges at $v$ to $\Pi_{k}$. Do the same for the graph $G'$.

Define a graph $H$ as follows ($v$ and $v'$ will refer to vertices of
$G$ and $G'$ respectively, and similarly for edges $e$, $e'$):
\begin{align*}
  \label{eq:2}
  \Vert(H)&:=\{(i,v,v',\alpha): v,v'\text{ of color }i,~ \alpha\in \Alpha_i\}\\
\Edge(H)&:=\{(k,e,e',\beta): e, e'\text{ of color }k,~ \beta\in \Beta_{k}\} \\
%\end{align*}
%with
%$$
\tail(k,e,e',\beta)&:=\bigl(\tail k, \tail e,\tail e', 
\phi_{k}(\psi_{vk}(e)\psi_{v'k}(e')^{-1},\beta)\bigr)\\
%$$ and $$
\overline{(k,e,e',\beta)}&:=(\bar k,\bar e,\bar e',\beta),\quad\text{so}\\
%\end{align*}
%so
\head(k,e,e',\beta)&=\bigl(\head k, \head e,\head e', 
\phi_{\bar k}(\psi_{v\bar k}(\bar e)\psi_{v'\bar k}(\bar
e')^{-1},\beta)\bigr)\,.
\end{align*}

We claim the obvious map $H\to G$ is a covering. So let $v$ be a
$i$--vertex of $G$ and $e$ a $k$--edge at $v$ and $(i,v,v',\alpha)$ a
vertex of $H$ lying over $v$. We must show there is exactly one edge
of $H$ at this vertex lying over $e$. The edge must have the form
$(k,e,e',\beta)$ with
$\phi_{k}(\psi_{vk}(e)\psi_{v'k}(e')^{-1},\beta)=\alpha$. Since $\phi_{k}$ is
a bijection, this equation determines $\beta$ and
$\psi_{vk}(e)\psi_{v'k}(e')^{-1}$ uniquely, hence also $\psi_{v'k}(e')$, which
determines $e'$. This proves the claim.
By symmetry, $H$ also covers $G'$, so the colored Leighton's theorem
is proved.\end{proof}

\begin{remark}\label{rk:leighton}
  Leighton's original proof is essentially the above proof with $A_i$
  the cyclic group $\Z/a_i$, $B_k$ its cyclic subgroup of order $b_k$,
  and $\Pi_k$ the quotient group $A_i/B_k\cong \Z/r_k$.
\end{remark}

\section{Symmetry-restricted graphs}
\label{sec:symmetry-restricted}

We define a concept of a ``{symmetry-restricted graph}.''  The
underlying data consist of a graph of colors $C$ together with a
collection, indexed by the vertex-colors $i\in I$, of finite
permutation groups $\Delta_i$ with an indexing of the orbits of
$\Delta_i$ by the edge colors $k$ with $\tail k=i$.
\begin{definition}\label{def:vertex restriction}
  A
\emph{symmetry-restricted graph} for these data is a $C$-colored graph
$G$ and for each $i$--vertex $v\in\Vert(G)$ a representation of
$\Delta_i$ as a color-preserving permutation group on the set
$\star(v)$ of edges departing $v$. A \emph{morphism} of
symmetry-restricted graphs $\phi\colon G\to G'$ is a colored graph
homomorphism $\phi$ which restricts to a weakly equivariant
isomorphism from $\star(v)$ to $\star(\phi(v))$ for each $v$. (A map
$\phi\colon X\to Y$ of $\Delta$--sets is \emph{weakly equivariant} if
it is equivariant up to conjugation, i.e., there is a $\gamma\in
\Delta$ such that $\phi(\delta z)=\gamma\delta\gamma^{-1}\phi(x)$ for
each $x\in X$ and $\delta \in\Delta$.) Note that a morphism is a
covering map; if it is bijective it is an \emph{automorphism}.
\end{definition}

An example of a symmetry-restricted graph in this sense is a
$k$-regular oriented fat graph; we have just one vertex color and the
group $\Delta$ is a cyclic group of order $k$ acting transitively on
each $\star(v)$.  Another example is the
following:
\begin{example}\label{ex:dc}
  Consider a ``graph'' $G$ in which each vertex is a small
  dodecahedron or cube, and each corner of a dodecahedron is connected
  by an edge to a corner of a cube and vice versa. The graph of colors
  is a graph $C$ with $\Vert(C)=\{d,c\}$, $\Edge(C)=\{e,\bar e\}$,
  $\tail e=d$, $\head e=c$. The groups $\Delta_d$ and $\Delta_c$
  are the symmetry groups of the dodecahedron and cube respectively,
  acting as permutation groups of the $20$ corners of the dodecahedron
  and the $8$ corners of the cube.  The graph $G$ is thus bipartite,
  with $20$ edges at each $d$--vertex and $8$ edges at each
  $c$--vertex.
\end{example}

The desired application in \cite{BN} of Leighton's theorem for
symmetry restricted graphs often needs a more general concept of
symmetry restriction,
which we describe, although we have no significant results for it.  We
first reformulate the definition of a symmetry restricted graph.

Let $T$ be any infinite tree and $\Gamma$ a subgroup of its full
automorphism group such that $C:=T/\Gamma$ is finite. For a vertex $v$
of $T$, denote by $\Gamma_{(v)}$ the restriction to the star of $v$ of
the vertex group $\Gamma_v$ (the isotropy group of $v$); this is a
finite permutation group on the edges at $v$ which we call the
\emph{restricted vertex group} for the action.  Up to isomorphism as a
permutation group, this group only depends on the image of $v$ in $C$,
so it can be taken as the datum $\Delta_i$ for symmetry restriction,
where $i\in C$ is the color of $v$.

Now assume that $\Gamma$ is maximal among groups that act on $T$ with
quotient $C$ and with prescribed restricted vertex groups. Then
$\Gamma$ is the symmetry restricted automorphism group $\Aut^s(T)$ of
$T$ (the superscript reminds to consider only symmetry restricted
automorphisms).  Any $T/\Gamma_0$, where $\Gamma_0\le
\Gamma=\Aut^s(T)$ is a subgroup which acts freely on $T$, is a
symmetry restricted graph for the given data.

Now, for an edge $e$ of $T$ define the \emph{restricted edge group}
$\Gamma_{(e)}$ to be the restriction to the star of $e$ ($e$ together
with adjacent edges) of the isotropy group $\Gamma_e$. The version
of symmetry restriction needed in \cite{BN} is as follows:
\begin{definition}\label{def:edge restriction}
  Suppose $\Gamma$ is maximal among groups that act on $T$ with
  quotient $C$ and with restricted vertex \emph{and} edge groups equal
  to those of $\Gamma$. The quotients $T/\Gamma_0$, where $\Gamma_0\le
  \Gamma$ acts freely on $T$, are symmetry restricted graphs for the
  data $(T,C,\Gamma)$.
\end{definition}

Note that $\Gamma_{(e)}$ is a subgroup of ${(\Gamma_{(\tau
    e)})}_e\times {(\Gamma_{(\iota e)})}_{\bar e}$, or
$({(\Gamma_{(\tau e)})}_e\times {(\Gamma_{(\iota e)})}_{\bar
  e})\rtimes \Z/2$ if $\Gamma$ inverts $e$.  The ``vertex-only''
definition of symmetry restriction (Definition \ref{def:vertex
  restriction}) is the special case when the restricted edge groups
are as large as possible: ${(\Gamma_{(\tau e)})}_e\times
{(\Gamma_{(\iota e)})}_{\bar e}$ or $({(\Gamma_{(\tau e)})}_e\times
{(\Gamma_{(\iota e)})}_{\bar e})\rtimes \Z/2$.

We'd like to know if Leighton's theorem extends to this setting. More
generally one could ask if Leighton's theorem extends when symmetry is
restricted on possibly larger finite portions of $T$. Unfortunately,
the only case in which we can give any answers is the ``vertex-only''
version (Definition \ref{def:vertex restriction}) of symmetry
restriction.

\begin{theorem}\label{th:symmetry-restricted} For the
  ``vertex-only'' version of symmetry restriction suppose the graph of
  colors is a tree. Then any two finite symmetry-restricted graphs $G$
  and $G'$ which have a common covering have a common finite covering.
\end{theorem}

For Example \ref{ex:dc} one has a simple geometric proof similar to
the fat-graph proof for $k$--regular graphs. Create a
$3$--dimensional ``fat graph'' from $G$ by truncating the corners of
the dodecahedra and cubes to form small triangles and thickening each
edge of $G$ to be a rod with triangular cross-section joining these
triangles. The rods should have a fixed length and thickness, and be
attached rigidly to the truncated polyhedra which play the r\^ole of
vertices. Then the universal covering is a 3-dimensional fat tree
whose isometry group acts properly discretely, so the result follows
as before.

But if we have a graph $G$ made, say, of icosahedra connected to cubes
by edges, then it is less obvious how to create a rigid fat-graph
version, since the vertex degrees of icosahedron and cube are $5$ and
$3$, which do not match.

To prove the above theorem we will need the graph of groups approach
of Bass and Kulkarni.

\section{The Bass-Kulkarni proof}\label{sec:bk}

We give a simplified version of the Bass-Kulkarni proof of Leighton's
theorem in its colored version, Theorem \ref{th:colored}.

We retain the notation of Section \ref{sec:leighton}. In particular,
$C=\tilde G/\Aut(\tilde G)$ is the (refined) graph of colors, with
vertex set $I$ and edge set $K$.  For the moment we assume that
$\Aut(\tilde G)$ acts without inversions, so $C$ has no edge with
$k=\bar k$.

We use this graph as the underlying graph for a graph of groups,
associating a group $\Alpha_i$ of size $a_i$ to each vertex $i$ and a
group $\Beta_k=\Beta_{\bar k}$ of size $b_k$ to each edge $k$, along
with an injection $\phi_k\colon \Beta_k\to \Alpha_{\tail k}$. Of
course we have to choose our groups so that $\Beta_k$ embeds in
$\Alpha_{\tail k}$ for each $k$; one such choice is the one of Remark
\ref{rk:leighton}.

Let $\Gamma$ be the fundamental group of this graph of groups and $T$
the Bass-Serre tree on which $\Gamma$ acts; this action has quotient
$T/\Gamma=C$, vertex stabilizers $\Alpha_i$, and edge stabilizers
$\Beta_k$. Then $T$ is precisely the tree $\tilde G$. Now $\Gamma$
acts properly discretely on $T$. So, if we can express $G$ and $G'$ as
quotients $T/\Lambda$ and $T/\Lambda'$ with $\Lambda$ and $\Lambda'$
in $\Gamma$, then $\Lambda$ and $\Lambda'$ are finite index in
$\Gamma$ so $T/(\Lambda\cap\Lambda')$ is the desired common covering.

To complete the proof we must show that such $\Lambda$ and $\Lambda'$
exist in $\Gamma$. This is the content of Bass's Conjugacy Theorem
(\cite{B}, see also \cite{LT}).  We replace it for now by a ``fat
graph'' argument (but see Theorem \ref{th:conjugacy theorem}).

For each finite group $\Delta$ choose a finite complex $B\Delta$ with
fundamental group $\Delta$ and denote its universal covering by
$E\Delta$. We also assume that any inclusion $\phi\colon \Phi\to
\Delta$ of finite groups which we consider can be realized as the
induced map on fundamental groups of some map $B\phi\colon B\Phi\to
B\Delta$ (this is always possible, for example, if $B\Phi$ is a
presentation complex for a finite presentation of $\Phi$).  We now
create a ``fat graph'' version of our graph of groups by replacing
vertex $i$ by $B\Alpha_i$, edges $k$ and $\bar k$ by $B\Beta_k\times
[0,1]$ (with the parametrization of the interval $[0,1]$ reversed when
associating this to $\bar k$), and gluing each $B\Beta_k\times [0,1]$
to $B\Alpha_{\tail k}$ by the map $B\phi_k\colon
B\Beta_k\times\{0\}\to B\Alpha_{\tail k}$ which realizes the inclusion
$\phi_k\colon \Beta_k\to \Alpha_{\tail k}$.

This is a standard construction which replaces the graph of groups by
a finite complex $K$ whose fundamental group is $\Gamma$.  The
universal covering of $K$ is a fat-graph version $\mathcal T$ of the tree
$T$, obtained by replacing $i$--vertices by copies of $E\Alpha_i$ and
$k$--edges by copies of $E\Beta_k\times [0,1]$.  The ``fat edges''
$E\Beta_k\times [0,1]$ are glued to the ``fat vertices'' $E\Alpha_i$
by the lifts of the maps $B\phi_k$.  An automorphism of $\mathcal T$
will be a homeomorphism which is an isomorphism on each piece
$E\Alpha_i$ and $E\Beta_k\times [0,1]$ (where the only isomorphisms
allowed on an $E\Delta$ are covering transformations for the covering
$E\Delta\to K\Delta$).

We can similarly construct fat versions of the graphs $G$ and $G'$,
replacing each $i$--vertex by a copy of $E\Alpha_i$ and each $k$--edge
by a copy of $E\Beta_k\times [0,1]$.  There is choice in this
construction: if $\tail k=i$ then at the $E\Alpha_i$ corresponding to
an $i$--vertex $v$ there are $r_k$ edge-pieces $E\Beta_k\times [0,1]$
to glue to $E\Alpha_i$ and $r_k$ ``places'' on $E\Alpha_i$ to do the
gluing, and we can choose any bijection between these edge-pieces and
places; moreover, each gluing is then only determined
up to the action of $\Beta_k$.  Nevertheless, however we make these
choices, we have:
\begin{lemma}\label{le:Conjugacy Theorem}
  The above fattened graphs have universal covering isomorphic to
  $\mathcal T$.
\end{lemma}
\begin{proof} We construct an isomorphism of $\mathcal T$ to the universal
  covering of the fattened $G$ inductively over larger and larger finite
  portions.  The point is that if one has constructed the isomorphism
  on a finite connected portion of $\mathcal T$, when extending to an adjacent
  piece (either an $E\Alpha_i$ or an $E\Beta_k\times [0,1]$), the
  choice in the gluing map for that piece is an element of a
  $\Beta_k$, which extends over the piece, so the isomorphism can be
  extended over that piece.
\end{proof}

Since the fattened versions of $G$ and $G'$ each have universal
covering $\mathcal T$, they are each given by an action of a subgroup of
$\Gamma=\Aut(\mathcal T)$, as desired, thus completing the proof of the colored
Leighton's theorem for the case that $\Gamma$ has no inversions. 

If $\Gamma$ does invert some edge, so $C$ has an edge $k=\bar k$, then
the edge stabilizer is an extension $\bar\Beta_k$ of the cyclic group
$C_2$ by $\Beta_k$. We can assume that the inclusion
$\Beta_k\subset\bar\Beta_k$ is represented by a double covering
$K\Beta_k\to K\bar\Beta_k$. The complex
%$K\Beta_k\times_{C_2}[0,1]:=
$(K\Beta_k\times[0,1])/C_2$ (diagonal
action of $C_2$) is then the object that the ``half-edge'' $k$ of $C$
should be replaced by in fattening $C$.  The proof then goes through
as before.\qed

\smallskip 
Our earlier fat graph proofs for the $k$--regular case and for Example
\ref{ex:dc} are special cases of the proof we have just given if we
generalize the proof to 
allow orbifolds.

\section{Proofs for symmetry-restricted graphs}
\label{sec:symmetry-restricted-again}

\begin{proof}
  [Proof of Theorem \ref{th:symmetry-restricted}]
Recall the situation of Theorem \ref{th:symmetry-restricted}: we have
a graph of colors $C$ defining the set of vertex colors $I=\Vert(C)$,
and for each $i\in I$ we have a finite permutation group $\Delta_i$
which acts as a permutation group of $\star(v)$ for each $i$--vertex
$v$ of our colored graphs. Theorem \ref{th:symmetry-restricted} also
required the graph of colors $C$ to be a tree; for the moment we will
not assume this.

Consider an edge-color $k\in K=\Edge(C)$
with $\partial k=ij$. For a $k$--edge $e$ of $G$ the stabilizer $e$ in
the $\Delta_i$ action on $\star(\tail e)$ will be denoted $\Delta_k$;
it is a subgroup of $\Delta_i$ which is determined up to
conjugation, so we make a choice.

Suppose that for every $k$ we have $\Delta_k\cong \Delta_{\bar
  k}$. This is the case, for example, for the dodecahedron-cube graphs
of Example \ref{ex:dc}, where these stabilizer groups are dihedral of
order $6$.  In this case the proof of the previous section works with
essentially no change, using $\Alpha_i=\Delta_i$ and
$\Beta_k=\Delta_k$.  The only change is that when fattening the $r_k$
$k$--edges at a fattened $i$--vertex $v$, our freedom of choice in
attaching the $r_k$ edge-pieces $E\Beta_k\times [0,1]$ to $r_k$ places
on $E\Alpha_i$ is now restricted: we must attach them equivariantly
with respect to the action of $\Alpha_i=\Delta_i$ on $\star(v)$ (this
still leaves some choice). This proves \ref{th:symmetry-restricted}
for this case.

Now assume $C$ is a tree. We can reduce the general case to the above
special case as follows: For any vertex color $i$ define $K_i:=\{k\in
K: k$ points towards $i\}$ and then replace each $\Delta_i$ by
$\bar\Delta_i:=\Delta_i\times\prod_{k\in K_i}\Delta_k$, acting on
$\star(i)$ via the projection to $\Delta_i$. Then the stabilizer $\bar
\Delta_k$ is $\Delta_k\times\prod_{k\in K_{\tail k}}\Delta_k$, which
equals $\bar\Delta_{\bar k}$, so we are in the situation of the
previous case.\end{proof}

It is not hard to extend the above proof to prove the following
theorem, which we leave to the reader.
\begin{theorem}\label{th:symmetry-restricted2}
  Suppose that for every closed directed path $(k_1,k_2,\dots,k_r)$ in
  the graph of colors we have
  $\prod_{i=1}^r\Delta_{k_i}\cong\prod_{i=1}^r\Delta_{\bar k_i}$ (note
  that these groups have the same order). Then any two finite
  symmetry-restricted graphs with these data which have a common covering
  have a common finite covering.\qed
\end{theorem}

One of the two ingredients of the original Bass-Kulkarni proof of
Leighton's theorem is Bass's Conjugacy Theorem. This theorem holds for
symmetry-restricted graphs (see below), but this appears not
to help extend the above results. The other ingredient in the
Bass-Kulkarni proof is to find a subgroup of $\Aut(T)$ which acts
properly discretely on $T$ with quotient $C$. Such a group would
necessarily be given by a graph of finite groups with underlying graph
$C$, and we are back in the situation of the proof we have already
given, which appears to need strong conditions on $C$. 

The Conjugacy Theorem says, in our language, that if $T$ is a colored
tree whose colored automorphism group acts without inversions (i.e.,
the graph of colors $C=T/\Aut(T)$ has no loops $k=\bar k$), and
$H\subset \Aut(T)$ is a subgroup with $T/H=C$, then any $\Gamma$ which
acts freely on $T$ can be conjugated into $H$ by an element of
$\Aut(T)$.  In the symmetry restricted setting we  write
$\Aut^s(T)$ to remind that we mean symmetry restricted automorphisms.
\begin{theorem}[Conjugacy Theorem]\label{th:conjugacy theorem}
  Fix data for symmetry-restricted graphs (``vertex-only'' version),
  and assume the graph of colors $C$ has no loops. If\/ $T$ is the
  symmetry-restricted tree for this data (it is unique) and $H\subset
  \Aut^s(T)$ a subgroup with $T/H=C$, then for any $\Gamma\subset
  \Aut^s(T)$ which acts freely on $T$, there exists $g\in \Aut^s(T)$
  with $g\Gamma g^{-1}\subset H$.
\end{theorem}
\begin{proof} In \cite{B} Bass includes a short proof of his Conjugacy
  Theorem proposed by the referee. That proof constructs the
  conjugating element $g$ directly, and one verifies by inspection
  that $g$ is a symmetry restricted automorphism. The point is that
  $g$ is the identity on the stars of a representative set $S$ of
  vertices for orbits of the $\Gamma$ action. If $v$ is any vertex of
  $T$ let $\gamma\in \Gamma$ be the element that takes $v$ to a vertex
  in $S$ and $h=g\gamma g^{-1}\in H$. Restricted to the star of $v$
  the map $g$ is $h^{-1}\gamma$, which is in $\Aut^s(T)$.
\end{proof}
It is worth noting that the ``fat graph argument'' in Section
\ref{sec:bk} proves Bass's Conjugacy Theorem in its original form (see
also \cite{LT}). However, neither that proof nor Bass's proof can be
applied to the symmetry restricted case, so it is somewhat remarkable
that the above proof works.

The above approach to extend Leighton's theorem involved extending
each vertex group to make it act non-effectively on the star of the
vertex; we used trivial extensions (direct products). By using other
extensions we can prove isolated additional cases, but this approach
is immediately obstructed if $C$ has a closed directed path
$\{k_1,\dots,k_n\}$ for which
$\sum_{i=1}^r[\Delta_{k_i}]\ne\sum_{i=1}^r[\Delta_{\bar k_i}]$, where
$[\Delta]$ means the class of $\Delta$ in the Grothendieck group of
finite groups modulo the relations given by short exact sequences (so
a group is equivalent to the sum of its composition factors).

In the application to \cite{BN}, the groups $\Delta_k$ and
$\Delta_{\bar k}$ are both extensions of a finite cyclic group
$F_k=F_{\bar k}$ of order $1$, $2$, $3$, $4$ or $6$ by a finite
2-generator abelian group, so the above the obstruction does not
arise. Nevertheless, we have been unable to resolve the question in
general for this case, even if the $F_k$ are trivial. And when an
$F_k$ is non-trivial we have a corresponding edge restriction
(Definition \ref{def:edge restriction}),  namely the subgroup of
$(a,b)\in\Delta_k\times\Delta_{\bar k}$ for which $a$ and $b$ have the
same image in $F_k$, so we are outside the cases of symmetry
restriction where we have any results.

\end{document}